\documentclass[11pt,reqno]{amsart} 

\usepackage{graphicx}

\usepackage[top=1in,bottom=1in,left=1in,right=1in]{geometry}

\usepackage{xcolor}

\usepackage{enumitem}

\usepackage[labelfont=bf,font=sf]{caption}

\usepackage[font={sf,small},labelfont=md]{subfig}

\usepackage[noadjust,nosort]{cite}

\usepackage{array}
\usepackage{amsmath, amsthm, amssymb}
\usepackage{enumerate}
\usepackage{tikz}
\usepackage{mathtools}

\usepackage{polynom}

\usepackage{multicol}

\usepackage{dsfont}
	\newcommand{\one}{\mathds{1}}

\usepackage{framed}

\usepackage{hyperref}
	\hypersetup{colorlinks=true,linkcolor=blue,citecolor=blue,urlcolor=black}
	
\allowdisplaybreaks

\numberwithin{equation}{section}


\newcommand{\eq}[1]{\begin{align*} #1 \end{align*}}
\newcommand{\eeq}[1]{\begin{align} \begin{split} #1 \end{split} \end{align}}

\newcommand{\N}{\mathbb{N}}

\renewcommand{\P}{\mathbb{P}}

\renewcommand{\a}{\mathcal{A}}

\newcommand{\s}{\mathcal{S}}

\newtheorem{thm}{Theorem}[section]

\newtheorem{lemma}[thm]{Lemma}

\theoremstyle{definition}
\newtheorem{defn}[thm]{Definition}

\newtheorem{remark}[thm]{Remark}

\newcommand{\vc}[1]{{\mathbf #1}}

\DeclarePairedDelimiter\ceil{\lceil}{\rceil}

\title{An upper bound on the size of avoidance couplings}



\author{Erik Bates}
\author{Lisa Sauermann}
\thanks{Erik Bates's research was partially supported by NSF grant DGE-114747} 
\address{\newline Department of Mathematics \newline Stanford University \newline 450 Serra Mall, Bldg 380 \newline Stanford, CA 94305 \newline \textup{\tt ewbates@stanford.edu}, \textup{\tt lsauerma@stanford.edu}}

\begin{document}
\bibliographystyle{acm}

\begin{abstract}
We show that a coupling of non-colliding simple random walkers on the complete graph on $n$ vertices can include at most $n - \log n$ walkers.
This improves the only previously known upper bound of $n-2$ due to Angel, Holroyd, Martin, Wilson, and Winkler ({\it Electron.~Commun.~Probab.~18},  2013).
The proof considers couplings of i.i.d.~sequences of Bernoulli random variables satisfying a similar avoidance property, for which there is separate interest.
\end{abstract}

\maketitle

\section{Introduction and main results} \label{intro}
The notion of an avoidance coupling was introduced by Angel, Holroyd, Martin, Wilson, and Winkler \cite{angel-holroyd-martin-wilson-winkler13}.
Consider simple random walk on a finite graph $G$, in which a walker moves at each step by choosing an adjacent vertex uniformly at random.
The goal of an avoidance coupling is to couple two or more random walkers, who are restricted to move one at a time in cyclical order, in such a way that they never meet. 
At the same time, each walker must be performing a simple random walk on $G$ when viewed separately from the other walkers.
Clearly if an avoidance coupling of $k$ walkers can take place on a given graph, then so can one of $k-1$ walkers, simply by making one of the walkers invisible.
So given $G$, a natural question is, ``What is the maximum number of walkers in an avoidance coupling on $G$?"

Like Angel \textit{et al.}, we restrict ourselves to the case when $G = K_n$, the complete graph on $n$ vertices, or $G=K_n^*$, the complete graph with a loop at every vertex.
In this setting, the construction of avoidance couplings is made easier by symmetry, but also made harder by the great freedom each walker must have in moving.
Nevertheless, the authors of \cite{angel-holroyd-martin-wilson-winkler13} construct an avoidance coupling of $k$ walkers on $K_{2k+1}$, $K_{2k+1}^*$, and $K_{2k}^*$, when $k$ is a power of $2$.
This established a linear number of walkers on the complete graphs $K_n$ and $K_n^*$, but only for specific values of $n$.
In the looped case, though, they were able to prove that an avoidance coupling of $k$ walkers on $K_n^*$ could be extended to one on $K_{n+1}^*$, thereby making possible $\ceil{n/4}$ walkers on any $K_n^*$.
Later Feldheim \cite{feldheim17} proved this monotonicity principle in the loopless case, thereby extending the same linear lower bound to $K_n$.

In the way of upper bounds, less progress has been made.
The only previously known result \cite[Theorem 8.1]{angel-holroyd-martin-wilson-winkler13} says that there is no avoidance coupling of $n-1$ walkers on $K_n^*$ for $n\geq4$. 
This observation is trivial on the loopless graph $K_n$, where each of $n-1$ walkers could only move deterministically to the open site.
It would be satisfying to prove a linear upper bound of $cn$ with $c < 1$, as Angel \textit{et al.}~propose in \cite[Section 9]{angel-holroyd-martin-wilson-winkler13}.
A more manageable task, however, is to produce an upper bound $U(n)$ such that $n-U(n) \to \infty$, which we are able to do in Theorem \ref{ac_thm} below.
The result is obtained by considering couplings satisfying a weaker avoidance property described below, for which our  bound may be closer to the truth.

For a fixed $p \in (0,1)$, we say that $k$ coupled walkers on $K_2^* = \{0,1\}$ form a \textit{$1$-avoidance} if, while taking turns in cyclical order as before, no two walkers simultaneously occupy site $1$ but each walker's trajectory forms a sequence of i.i.d.~Bernoulli($p$) random variables.
That is, any given walker can be found at site $1$ with probability $p$ at any given turn, independent of all other turns, but only site $0$ can accommodate more than one walker.
For this scenario, we have the following result.

\begin{thm} \label{1a_thm}
If there exists a $1$-avoidance coupling of $k$ Bernoulli($p$) walkers, then $p(1-p\log p) \leq 1/k$.
\end{thm}

We give the proof in Section \ref{proofs} after providing an overview of relevant literature in Section \ref{background}.
The connection between standard avoidance couplings and $1$-avoidance couplings is seen by tracking which walker on $K_n^*$, if any, currently occupies vertex $1$.
Any given simple random walker on $K_n^*$ will be found at vertex $1$ will probability $1/n$, independently of that walker's position at all other times.
Therefore, if there exists an avoidance coupling of $k$ walkers on $K_n^*$, then there is a $1$-avoidance coupling of $k$ Bernoulli($1/n$) walkers, an observation originally made in \cite[Lemma 5.2]{angel-holroyd-martin-wilson-winkler13}.
Therefore, Theorem \ref{1a_thm} easily produces the following result, as shown in Section \ref{proof_2}. 

\begin{thm} \label{ac_thm}
For any $n \geq 3$, an avoidance coupling on $K_n^*$ (and therefore on $K_n$) can have at most $\ceil{n-\log n}$ walkers.
\end{thm}

\begin{remark}
Here $\log n$ denotes the natural logarithm.
The upper bound $\ceil{n-\log n}$ is strictly less than $n-2$ as soon as $n \geq 21$.
\end{remark}

\begin{remark}
If there exists an avoidance coupling of $k$ walkers on the loopless graph $K_n$, then one of the same size exists on $K_n^*$. 
To see this fact, one can modify any coupling on $K_n$ to one on $K_n^*$ in the following way.
Immediately before each turn of the first walker, it is decided independently with probability $1/n$ that all walkers will stay stationary during the coming round.
The authors of \cite{angel-holroyd-martin-wilson-winkler13} call this modification ``staying in waves".
\end{remark}

Given that $1$-avoidance is a weaker notion, it is natural to expect that the conclusion of Theorem \ref{ac_thm} is not optimal.
Nevertheless, there is separate interest in determining for a fixed positive integer $k$ the largest $p \in (0,1)$ such that $k$ Bernoulli($p$) walkers can be coupled in a $1$-avoidance.
Theorem \ref{1a_thm} provides an upper bound.
In the other direction, \cite[Section 5]{angel-holroyd-martin-wilson-winkler13} shows that if $k \leq \ceil{n/4}$, then any $p \leq 1/n$ is possible.

\section{Background} \label{background}
In this section we highlight several other questions concerning avoidance couplings, as well as two adjacent families of problems.

\subsection{Markovian avoidance couplings} \label{Markovian}
The only published works on avoidance couplings are due to Angel \textit{et al.}~\cite{angel-holroyd-martin-wilson-winkler13}, Feldheim \cite{feldheim17}, and Infeld \cite{infeld16}.
The first two deal exclusively with the case $G = K_n$ or $K_n^*$, and all three give special consideration to couplings that satisfy some type of Markov property.
Following the terminology from \cite{angel-holroyd-martin-wilson-winkler13}, which considers the strongest such property, we say that an avoidance coupling is \textit{Markovian} if the probability distribution of any particular walker's next move is entirely determined by the current configuration of walkers.
The simplicity of the dynamics makes constructing Markovian couplings a greater challenge; indeed, it is not even clear if a Markovian coupling of $k$ walkers always yields one of $k-1$ walkers.
Furthermore, while the aforementioned constructions in \cite{angel-holroyd-martin-wilson-winkler13}---for the special case when $k$ is a power of $2$---are Markovian, it is not known whether they can be extended to larger $n$ without losing this property.
Consequently, it remains an open problem to construct Markovian avoidance couplings on general $K_n$ or $K_n^*$ with a linear number of walkers.

In \cite[Theorem 7.1]{angel-holroyd-martin-wilson-winkler13}, Angel \textit{et al.}~construct Markovian  avoidance couplings of $k \leq n/(8\log_2 n)$ walkers on $K_n^*$, and $k\leq n/(56\log_2 n)$ walkers on $K_n$.
Dropping the Markovian condition, they establish $k \leq \ceil{n/4}$ walkers on $K_n^*$ from extensions of the special cases, 
as mentioned in Section \ref{intro}.
Feldheim \cite[Theorem 1.1]{feldheim17} shows that extensions can also be done for $K_n$, thus allowing $k \leq \ceil{n/4}$ walkers for general $n$, but still without preserving the Markovian property.
These extensions do, however, preserve a weaker property which Feldheim calls \textit{label-Markovian}.
In a label-Markovian avoidance coupling, the walkers need to agree on a random labeling of the vertices at the start of each round, in addition to examining the current configuration at their turns.

A yet weaker Markov property is studied in the thesis of Infeld \cite{infeld16}.
We will say an avoidance coupling is \textit{round-Markovian} if the configuration at the start of each round is a Markov process.
That is, the joint update of all walkers from round to round is Markovian, although individual walkers' moves may not be.
By considering avoidance couplings on graphs other than $K_n$ and $K_n^*$, Infeld is able to introduce a notion not possible on complete graphs: a \textit{uniform} avoidance coupling.
A round-Markovian avoidance coupling is said to be uniform if the transition probabilities for each walker do not depend on the configuration of the other walkers at the start of the round.
In \cite[Chapter 2]{infeld16}, the reader can find a breadth of examples as well as partial characterizations of graphs admitting a uniform avoidance coupling of two walkers.

\subsection{Applications}
Potential applications of avoidance couplings include scenarios in which multiple users want to make decisions based on a random walk, but need to avoid affecting other users.
For instance, a pollster sampling a large population over time may use a random walk to select people to survey.
If several pollsters work simultaneously to increase data collection, these walks must be coordinated to avoid repeated sampling at any given time.
As another example, several background applications might run simultaneously on a computer and access random parts of the hard drive, yet for speed or corruption reasons, they should not access any particular part at the same time \cite{tsianos13}.
Finally, in communication systems it can be necessary for users to periodically change transmission frequencies in order to counteract malicious attempts at interference or interception.
These updates should be random so as to not be predictable, but also separate messengers will not be able to transmit over the same frequency.

In each of these examples, it might further be desirable to make decisions independent of history, either for practicality (e.g.~the first example above), for conservation of memory (the second), or to avoid becoming more predictable as time goes on (the third).
For this reason, Markovian avoidance couplings are of special interest.
Moreover, there are simple cases for which an avoidance coupling exists, but a Markovian version does not (e.g.~see \cite[Theorem 3.1]{angel-holroyd-martin-wilson-winkler13}).

\subsection{Related problems} \label{related}
The task of keeping apart random walkers has also been studied in the context of scheduling problems, which frequently appear in computer science.
In this setting, the interest is solely in avoiding collisions rather than in also maintaining the law of a random walk. 
The moves of independent walkers, usually two, can be delayed by a scheduler, although collisions are still inevitable \cite{coppersmith-tetali-winkler93,tetali-winkler93} unless the scheduler is clairvoyant and knows the full future of both walkers \cite{winkler00,balister-bollobas-stacey00,gacs11,basu-sidoravicius-sly??} (at least on large enough complete graphs).

A problem of a slightly different flavor is that of co-adapted reflected Brownian motions that are ``shy", i.e.~remain a fixed positive distance apart.
The nonexistence of such a process has been established on bounded domains with various regularity properties \cite{bejamini-burdzy-chen06,kendall09,bramson-burdzy-kendall13,bramson-burdzy-kendall14}.

\section{Proofs of Theorems \ref{1a_thm} and \ref{ac_thm}} \label{proofs}
We now prove Theorem \ref{1a_thm} and deduce Theorem \ref{ac_thm} as an easy corollary.

\subsection{Preliminaries} \label{preliminaries}
Let us begin by establishing some notation.
We use the conventions $\N \coloneqq \{1,2,\dots\}$ and $[k] = \{1,\dots,k\}$ for a fixed positive integer $k$.

\begin{defn}
A $1$-avoidance coupling of $k$ Bernoulli($p$) walkers is a $\{0,1\}^k$-valued process $\vc X = (\vc X(t))_{t\in\N} = (X_1(t),\dots,X_k(t))_{t\in\N}$ such that
\begin{itemize}
\item[(i)] (\textit{faithfulness}) for each $i = 1,\dots,k$, $(X_i(t))_{t\in\N}$ is a sequence of i.i.d.~Bernoulli($p$) random variables;
\item[(ii)] (\textit{avoidance}) for every $t \in \N$ and $1 \leq i < j \leq k$, we have $\P(X_i(t) = X_j(t) = 1) = \P(X_{i}(t+1) = X_j(t) = 1) = 0$.
\end{itemize}
\end{defn}

Note that property (ii) includes two conditions.  
First, at each time $t$, the vector $\vc X(t)$ can have at most one coordinate equal to $1$.
Second, at time $t+1$ either the location of that coordinate weakly increases toward $k$ or no coordinate is equal to $1$ (in the latter case, any coordinate is permitted to equal $1$ at time $t+2$).
Because of the first condition, a sample of the coupling can almost surely be represented by an infinite sequence $(a(t))_{t\in\N}$ of characters in the alphabet $\a \coloneqq [k] \cup \{\texttt{B}\}$, where \texttt{B} is the ``blank" placeholder for times at which no coordinate is equal to $1$.
More precisely, we set $a(t) = j$ if and only if $X_j(t) = 1$, and $a(t) = \texttt{B}$ if and only if $X_i(t) = 0$ for all $i\in[k]$.
Then the latter condition translates to the following analog.

\begin{defn} \label{permissible_defn}
We say that a sequence $(a(t))_{t\in\N} \in \a^\N$ is \textit{permissible} if $a(t) \leq a(t+1)$ whenever $a(t),a(t+1) \in [k]$.
\end{defn}

Once we encode a $1$-avoidance coupling as a random, almost surely permissible sequence, it will be useful to consider the ``gaps" between successive occurrences of a given symbol $i \in [k]$.
Indeed, such an occurrence appears with probability $p$ at any given time, independently of all other times.
We thus make the following definitions.

\begin{defn} \label{weight_defn}
Let $\s = (a(t))_{t=1}^T \in \a^\N$ be a finite sequence in the alphabet $\a$.
\begin{itemize}
\item[(a)] We say $t_1 < t_2$ are \textit{neighbors} if there is $i \in [k]$ such that $a(t_1) = a(t_2) = i$ and $a(t) \neq i$ for all $t_1 < t < t_2$.
\item[(b)] The \emph{weight} of a pair of neighbors $t_1 < t_2$ is defined to be $1/b$, where $b \leq k$ is the number of distinct elements of $\a$ appearing in $(a(t))_{t_1 < t < t_2}$.
If $b=0$, we simply set the pair's weight to $0$.
\item[(c)] The \emph{total weight} of $\s$ is the sum of the weights of all pairs of neighbors.
\end{itemize}
\end{defn}

For example, consider the sequence displayed below with $k = 3$ and $T=15$.
There are four pairs of neighbors associated with the symbol ``$3$".
The first such pair is separated by two distinct characters, ``\texttt{B}" and ``$2$", and thus has weight $1/2$.
The second pair has weight $0$ since the two instances of ``3" are immediately adjacent.
The third pair has weight $1$ since the two neighbors are separated by just ``\texttt{B}" (notice that in a permissible sequence such as the one below, the only way for a pair of neighbors to have weight $1$ is to be separated by a string of \texttt{B}'s).
The fourth pair is separated by six characters, although among these there are only three distinct symbols; hence the weight of this final pair is $1/3$.

\begin{table}[!h]
\begin{tabular}{>{$}c<{$}|>{$}c<{$}>{$}c<{$}>{$}c<{$}>{$}c<{$}>{$}c<{$}>{$}c<{$}>{$}c<{$}>{$}c<{$}>{$}c<{$}>{$}c<{$}>{$}c<{$}>{$}c<{$}>{$}c<{$}>{$}c<{$}>{$}c<{$}}
t & 1 & 2 & 3 & 4 & 5 & 6 & 7 & 8 & 9 & 10 & 11 & 12 & 13 & 14 & 15 \\ \hline
a(t) & 1 & 3 & \texttt{B} & 2 & 3 & 3 & \texttt{B} & 3 & \texttt{B} & 1 & \texttt{B} & 2 & \texttt{B} & 1 & 3
\end{tabular}
\end{table}

The crucial lemma for establishing Theorem \ref{1a_thm} is the following.

\begin{lemma} \label{key_lemma}
Let $\s = (a(t))_{t=1}^{T}$ be any finite, permissible sequence in the alphabet $\a$.
Then the total weight of $\s$ is at most the number of \texttt{B}'s in $\s$.
\end{lemma}

\begin{proof}
The argument proceeds by induction on the length $T$ of the sequence. 
For sequences of length one, the statement is trivially true (since there are no pairs of neighbors). 
So henceforth fix a permissible sequence $\s$ and assume the claim has been demonstrated for all shorter sequences. 
We may assume without loss of generality that each of the numbers $1,\dots,k$ occurs at least once in $\s$; otherwise we can simply work on a smaller alphabet.
Furthermore, replacing any instance of \texttt{BB} with just \texttt{B} does not change the total weight of $\s$, and so we may assume $\s$ contains no consecutive \texttt{B}'s.

First consider the case when $\s$ contains a pair of neighbors of weight $0$, i.e.~$a(t) = a({t+1}) = i$ for some $t$ and $i \in [k]$.
By deleting one of these identical characters from the sequence, we obtain a shorter sequence with the total weight unchanged; see Figure \ref{lemma_figure}(b).
Therefore, this case follows by the induction hypothesis, and we may henceforth assume that no two adjacent characters in $\s$ are identical.
Since Definition \ref{permissible_defn} forces the numeric characters of $\s$ to be increasing until $\texttt{B}$ appears, a consequence of this assumption is the following: the symbol $\texttt{B}$ appears at least once between any pair of neighbors.

Next suppose the sequence $\s$ contains a pair of neighbors of weight $1$, i.e. $a(t) = a({t+2}) = i$ and $a({t+1}) = \texttt{B}$.
In this case we replace the pattern $i\texttt{B}i$ by $i$, as in Figure \ref{lemma_figure}(c).
It is easy to see that this modification decreases the total weight by exactly $1$ (one neighbor-pair of $i$'s of weight $1$ has disappeared, while all other weights are unchanged). On the other hand, the number of $\texttt{B}$'s in the sequence has also decreased by one.
As the resulting sequence is shorter, the desired statement again follows by induction.

From the previous two paragraphs, we may assume that $\s$ contains neither a pair of neighbors with weight $0$ or $1$. 
That is, for each pair of neighbors the number $b$ in the definition of weight satisfies $b\geq 2$.
We can now calculate the total weight of $\s$ in two different ways.
On one hand, for each $i \in [k]$ we can simply sum all the weights of neighbor-pairs of $i$'s.
Let us call this quantity the \textit{output} of $i$, and then the total weight of $\s$ is the sum of the outputs.

Let us next imagine a second way of summing weights. 
Consider any pair of neighbors in $\s$. 
Let $b$ denote the number of distinct symbols occurring in between, so that the weight of the neighbor-pair is $1/b$. 
Recall that one of the symbols must be \texttt{B} by our earlier assumption, hence there are exactly $b-1$ different numbers appearing between the two neighbors. 
Now imagine that the pair of neighbors ``donates" weight
$1/(b(b-1))$ to each of these $b-1$ different numbers (note that these donated weights sum precisely to the pair's weight, $1/b$). 
Imagine this redistribution happens for all pairs of neighbors so that all the weights from the neighbor-pairs are distributed among the numbers $1,\dots,k$. 
Note that a neighbor-pair consisting of two $i$'s does not donate any weight to the number $i$ itself. 
For each $i\in[k]$, let the \emph{input} of $i$ be the sum of all weights that $i$ receives in this redistribution process, e.g.~Figure \ref{lemma_figure}(d).

\begin{figure}[!ht]
\centering
\includegraphics[width=0.5\textwidth]{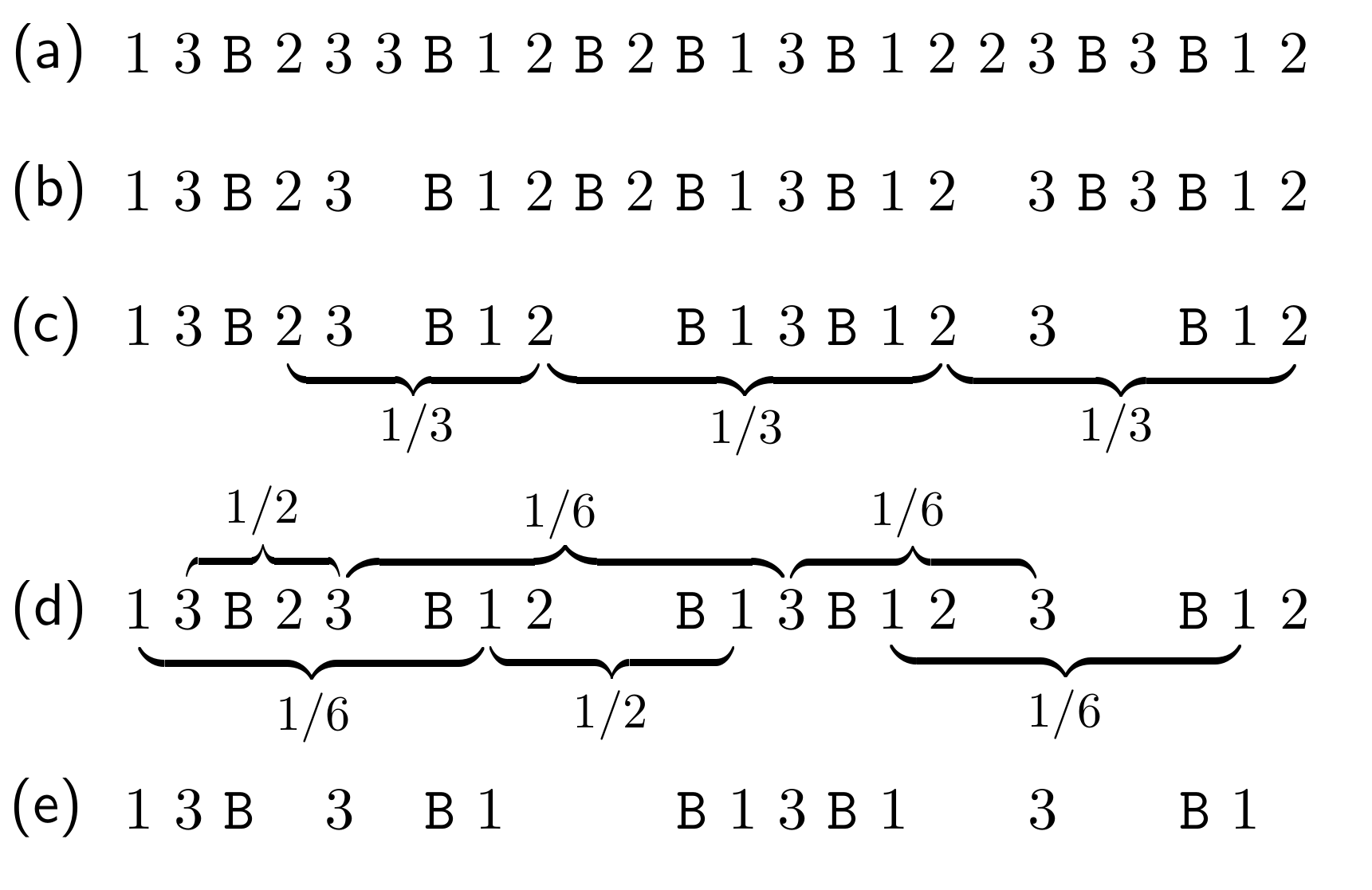}
\caption{{Reductions in the proof of Lemma \ref{key_lemma}.
\textbf{{(a)}} The original sequence $\s$.
\textbf{{(b)}} All weight-$0$ pairs removed.
\textbf{{(c)}} All weight-$1$ pairs removed; the output of $2$ is computed as $1/3+1/3+1/3 = 1$.
\textbf{{(d)}} The input of $2$ is computed as $2(1/2+1/6+1/6) = 5/3$.  Weights donated from $1$ are shown on bottom, those from $3$ on top.
\textbf{{(e)}} Since $5/3 > 1$, the symbol $2$ can be removed to produce the shorter permissible sequence $\s'$ with a greater total weight.
}}
\label{lemma_figure}
\end{figure}

By design, the sum of inputs is equal to the sum of weights over all neighbor-pairs, which in turn is equal to the sum of outputs.
Hence, there must be some $j\in [k]$ such that the input of $j$ is at least the output of $j$.
Let us now delete all $j$'s from the sequence $\s$, producing a new permissible sequence $\s'$. 
We claim that the total weight of $\s'$ is at least that of $\s$. 

The neighbor-pairs in $\s'$ are precisely the neighbor-pairs in $\s$ that do not consist of two $j$'s.  So let us consider any neighbor-pair in $\s$ consisting of two $i$'s with $i\neq j$. 
Let the weight of this neighbor-pair in $\s$ be $1/b$, i.e.~there are exactly $b$ different symbols occurring in between in $\s$. 
If none of these $b$ symbols is $j$, then the weight of the neighbor-pair in $\s'$ is also $1/b$. 
If instead $j$ is one of the symbols, then there will be only $b-1$ different symbols between the neighbor-pair in $\s'$.
So this pair's new weight in $\s'$ is $1/(b-1)$, constituting a increase in weight by $1/(b(b-1))$.
In each of the two cases, the weight of the neighbor-pair increases by the same amount said pair donated to $j$ in the weight redistribution process described above: in the first case $0$, in the second case $1/(b(b-1))$.
In summary, the sum of the weights of all neighbor-pairs in $\s'$ is equal to the sum of their weights in $\s$ plus the input of $j$.
More formally,
\eq{
\text{(total weight of $\s'$)} 
&= \sum_{i\in[k]\setminus\{j\}}\sum_{\text{neighbor-pairs of $i$'s}} (\text{weight of neighbor-pair in $\s'$}) \\
&=  \bigg(\sum_{i\in[k]\setminus\{j\}}\sum_{\text{neighbor-pairs of $i$'s}} (\text{weight of neighbor-pair in $\s$})\bigg) + (\text{input of $j$}) \\
&= (\text{total weight of $\s$}) - (\text{output of $j$}) + (\text{input of $j$}).
}
By our choice of $j$, the above equation proves the claim that the total weight of $\s'$ is at least that of $\s$. 

Since $\s$ was assumed to contain at least one occurrence of $j$, the sequence $\s'$ is strictly shorter than $\s$. 
Furthermore, the two sequences contain the same number of \texttt{B}'s, and so we are done by the above claim and induction.
\end{proof}

\subsection{Proof of Theorem \ref{1a_thm}}

Let $(\vc X(t))_{t\in\N}$ be a $1$-avoidance coupling of $k$ Bernoulli($p$) walkers.
Since each coordinate $i$ induces a sequence of i.i.d.~Bernoulli($p$) random variables, the law of large numbers gives
\eeq{ \label{LLN_1}
\lim_{T\to\infty}\frac{1}{T}\sum_{t=1}^{T} \one_{\{X_i(t) = 1\}} = p \quad \mathrm{a.s.}
}
Because no two coordinates can simultaneously equal $1$, we also have
\eq{
\one_{\bigcup_{i=1}^k \{X_i(t) = 1\}} = \sum_{i=1}^k \one_{\{X_i(t) = 1\}} \quad \mathrm{a.s.},
}
and therefore
\eeq{ \label{one_present}
\lim_{T\to\infty}\frac{1}{T}\sum_{t = 1}^{T} \one_{\bigcup_{i=1}^k \{X_i(t) = 1\}}
 = kp \quad \mathrm{a.s.}
}
Since
\eeq{ \label{either_case}
\frac{1}{T}\sum_{t=1}^{T}  \big(\one_{\bigcup_{i=1}^k \{X_i(t) = 1\}} + \one_{\bigcap_{i=1}^k \{X_i(t)=0\}}\big) = 1,
}
the observation \eqref{one_present} immediately gives the trivial bound $p \leq 1/k$.
The correction factor $(1-p\log p)$ will be obtained by considering the more complicated quantity
\eq{
Z \coloneqq \liminf_{T\to\infty} \frac{1}{T}\sum_{t=1}^{T}  \one_{\bigcap_{i=1}^k \{X_i(t)=0\}}.
}
In particular, we seek a lower bound on $Z$.
The rest of the proof is to establish a sufficiently strong bound.

Recall the definitions of Section \ref{preliminaries}.
Let $(a(t))_{t\in\N} \in \a^\N$ be the random sequence obtained from $(\vc X(t))_{t\in\N}$, where $\a = [k] \cup \{\texttt{B}\}$.
In particular, $a(t) = \texttt{B}$ if and only if $X_i(t) = 0$ for all $i \in [k]$,
and so $Z$ is the (lower) asymptotic fraction of times $t$ for which $a(t) = \texttt{B}$.
Consequently, Lemma \ref{key_lemma} suggests we analyze weights of neighbor-pairs.
We make this anaysis precise in the following way.

If $a(t) = j \in [k]$ but $t$ is not the first time $j$ appears, then $a(t)$ is the ``right neighbor" to some ``left neighbor" $a(s) = j$ with $s < t$ and $a(u) \neq j$ for all $s < u < t$.
Let $b$ denote the number of distinct symbols appearing between the two neighbors.
If $b \geq 1$, then let $w_j(t) = 1/b$ be the weight of this pair, and set $w_i(t) = 0$ for all $i \in [k] \setminus \{j\}$.
If either $a(t)$ is the first occurrence of $j$ in the sequence, or $a(t-1) = j$, or $a(t) = \texttt{B}$, then set $w_i(t) = 0$ for all $i\in[k]$.
By Lemma \ref{key_lemma},
\eeq{ \label{weight_ineq}
\sum_{t = 1}^T  \one_{\bigcap_{i=1}^k \{X_i(t)=0\}} = \sum_{t = 1}^T  \one_{\{a(t) = \texttt{B}\}} \geq \sum_{t=1}^T \sum_{i = 1}^k w_i(t),
}
where the rightmost sum is the total weight of the sequence $(a(t))_{t=1}^T$ as defined in Definition \ref{weight_defn}.
It is this quantity for which we can compute an asymptotic lower bound.

Fix $i \in [k]$, and let $t_0 < t_1 < \cdots$ be all the (random) times for which $a(t) = i$, or equivalently $X_i(t) = 1$.
Because $(X_i(t))_{t\in\N}$ is a sequence of i.i.d.~Bernoulli($p$) random variables, the differences $(t_\ell-t_{\ell-1})_{\ell\in\N}$ are also i.i.d., with
\eq{
\P(t_\ell-t_{\ell-1} = b+1) = p(1-p)^b, \qquad b = 0,1,2,\dots.
}
In particular, for each $b$ the law of large numbers guarantees
\eq{
\lim_{L\to\infty} \frac{1}{L}\sum_{\ell=1}^L \one_{\{t_\ell-t_{\ell-1}=b+1\}} = p(1-p)^b \quad \mathrm{a.s.}
}
That is, $p(1-p)^b$ is the limiting proportion of those $a(t) = i$ for which the previous $i$ appeared exactly $b+1$ time steps earlier.
As the limiting proportion of $t$ for which $a(t) = i$ is precisely $p$, as stated in \eqref{LLN_1},
it follows that
\eeq{
\lim_{T\to\infty} \frac{1}{T}\sum_{t=1}^T \one_{\bigcup_{\ell \geq 1} \{t = t_\ell, t_\ell-t_{\ell-1} = b+1\}} = p^2(1-p)^b \quad \mathrm{a.s.} \label{LLN_2}
}
Of course, if $t_\ell - t_{\ell-1} = b+1$ and $b \geq 1$, then $w_i(t_\ell) \geq 1/b$, since $a(t_{\ell-1}+1),\dots,a(t_{\ell-1}+b)=a(t_\ell-1)$ constitute at most $b$ distinct symbols.
Therefore, by \eqref{LLN_2} we have that for any positive integer $N$,
\eq{
\liminf_{T\to\infty} \frac{1}{T}\sum_{t=1}^T w_i(t) 
&\geq \liminf_{T\to\infty} \frac{1}{T}\sum_{t=1}^T \sum_{b = 1}^\infty b^{-1}\one_{\bigcup_{\ell \geq 1} \{t = t_\ell, t_\ell-t_{\ell-1} = b+1\}} \\
&\geq \sum_{b = 1}^N b^{-1} \lim_{T\to\infty} \frac{1}{T} \sum_{t=1}^T \one_{\bigcup_{\ell \geq 1} \{t = t_\ell, t_\ell-t_{\ell-1} = b+1\}} 
=\sum_{b=1}^N \frac{p^2(1-p)^b}{b} \quad \mathrm{a.s.}
}
By allowing $N$ to tend to infinity and making the Taylor series computation
\eq{
\sum_{b = 1}^\infty \frac{p^2(1-p)^b}{b} = -p^2\sum_{b=1}^\infty (-1)^{b+1}\frac{(p-1)^b}{b} = -p^2\log p,
}
we see
\eq{
\liminf_{T \to \infty} \frac{1}{T}\sum_{t=1}^T w_i(t) \geq -p^2\log p \quad \mathrm{a.s.}
}
Since this argument holds for each $i \in [k]$, we conclude 
\eq{
Z \stackrel{\mbox{\scriptsize\eqref{weight_ineq}}}{\geq} 
\liminf_{T\to\infty}\frac{1}{T} \sum_{t=1}^T \sum_{i=1}^k w_i(t) \geq 
\sum_{i = 1}^k \liminf_{T\to\infty}\frac{1}{T}\sum_{t=1}^T w_i(t) \geq -kp^2\log p \quad \mathrm{a.s.},
} 
which in turn gives $kp -kp^2\log p \leq 1$ by \eqref{one_present} and \eqref{either_case}.
\qed


\subsection{Proof of Theorem \ref{ac_thm}} \label{proof_2}
Suppose there exists an avoidance coupling of $k$ walkers on $K_n^*$.
By \cite[Lemma 5.2]{angel-holroyd-martin-wilson-winkler13}, it follows that there is a $1$-avoidance coupling of $k$ Bernoulli($1/n$) walkers.
By Theorem \ref{1a_thm}, we must have
\eq{
\frac{1}{n}\Big(1 + \frac{\log n}{n}\Big) \leq \frac{1}{k} \quad \Rightarrow \quad
k \leq \frac{n^2}{n+\log n} = n - \log n + \frac{\log^2 n}{n+\log n}.
}
To complete the proof, note that $(\log^2 n)/(n+\log n) \leq 1$ for $n\geq 3$.
\qed

\section*{Acknowledgments}
We thank Omer Angel, Sourav Chatterjee, Alex Dunlap, Jacob Fox, and Mark Perlman for useful discussions.
We are especially grateful to Ohad Feldheim and an anonymous referee, whose feedback on earlier drafts helped improve the exposition.

\bibliography{avoidance_couplings}

\end{document}